\documentclass[a4paper, 12pt, oneside]{article}

\usepackage[utf8]{inputenc}
\usepackage[hidelinks]{hyperref}
\usepackage{float}
\usepackage{vmargin}
\usepackage{enumitem}
\usepackage{colortbl}
\usepackage{mathrsfs}
\usepackage{amsmath, amsthm, amsfonts, amssymb}
\usepackage{titlesec}

\def\vs{\vspace{0.1cm}\\}

\def\far{\uparrow}
\def\fab{\downarrow}
\def\fd{\rightarrow}
\def\fizq{\leftarrow}
\def\iin{\mathcal{E}}
\def\fid{\leftrightarrow}

\newtheorem{thm}{Theorem}[section]
\newtheorem{lem}[thm]{Lemma}
\newtheorem{prop}[thm]{Proposition}
\newtheorem{defn}[thm]{Definition}
\newtheorem{cor}[thm]{Corollary}

\theoremstyle{remark}

\newtheorem*{note}{Note}

\titleformat{\section}
  {\normalfont\fontsize{14}{15}\bfseries}{\thesection}{1em}{}
\title{$\rho$-assoc and $\rho$-dist of \textit{wfs} and \textit{f} in $\Sigma$ and $\mathscr{L_{HA}}$-theory on 0-OL}
\author{Lukas Restrepo\vs
  \scriptsize Universidad de Antioquia, Carmen de Viboral, Antioquia, Colombia\\
  \scriptsize Universidad Católica de Oriente, Rionegro, Antioquia, Colombia\\
  \scriptsize lmauricio.restrepo@udea.edu.co \ \ lukas.restrepo9160@uco.edu.co
  \date{August 12, 2014}
}

\begin{document}
\maketitle

\abstract{In this paper we create peudo associativity ($\rho$-assoc) and peudo distributivity ($\rho$-dist) properties for not fundamental operators NFO $\fab$, $\far$, using two semantic rules, also we build the proofs for this result in Hilbert-Ackermann ($\mathscr{HA}$) axiomatic system, all this in the 0-order logic (0-OL) context.\vspace{0.2cm}}

\noindent
Keywords: 0-order logic, peudo associativity, peudo distributivity.\\
AMS classification: 03B05.

\section{Introduction}

In 0-OL exists classic results about basic properties of associativity and distributivity with $\lor, \land, \fd, \fid$ and $\neg$ operators
\cite{B2} these are a consequence of the semantic (truth tables \cite{B1}) and syntactic ($\mathscr{HA}$ axiomatic system), in this paper we show a new notion about the associative and distributive properties for not fundamental operators (NFO).

\begin{defn}[NFO, FO] Are binary operators \\
NFO are the operators $\fab$, $\far$, $\fizq$, $\oplus$ the negations forms of the FO \\
FO are the classic operators $\lor$, $\land$, $\fd$, $\fid$ 
\end{defn}

\begin{defn}[$\Sigma_{NFBO}, \ \Sigma_{FBO}$] Are languages \textsc{\cite{B2}}\\
$\Sigma_{NFO}$ is the languaje with NFO, monary operator $\neg$ and parentheses.\\
$\Sigma_{FO}$ is the languaje with FO, monary operator $\neg$ and parentheses.
\end{defn}
The NFO and FO have dual representations in a $\Sigma$ language.

\section{Semantic comparison}

\begin{defn} [$\mathscr{A'}$-\textit{wfs} of $\Sigma_{NFO}$] 
A $\mathscr{A'}$-\textit{wfs} is a recursive string of the 0-OL semantic balanced and structurally well formed with interpretation
\textsc{\cite{B1}} that has the following elements.
\begin{enumerate}[topsep=2pt, partopsep=2pt, itemsep=2pt, parsep=2pt]
\item 	Atoms: $p,q, \dots$ that represent statements
\item 	Symbols of $\Sigma_{NFO}$
\end{enumerate}
\end{defn}

\begin{defn} [$\mathscr{A}$-\textit{wfs} of $\Sigma_{FO}$] 
A $\mathscr{A}$-\textit{wfs} is a recursive string of the 0-OL semantic balanced and structurally well formed with interpretation
that has the following elements.
\begin{enumerate}[topsep=2pt, partopsep=2pt, itemsep=2pt, parsep=2pt]
\item Atoms: $p,q, \dots$ that represent statements
\item Symbols of $\Sigma_{FO}$
\end{enumerate}
\end{defn}

\begin{note} The $\neg$ operator changes the interpretation of 1 to 0 and viceversa.
\end{note}

\begin{defn}[Truth Table \cite{B1}] Graphical format for strings $\mathscr{A}$ or $\mathscr{A'}$, containing all possible values of interpretations of the atoms $\mathscr{I}(p,q, \dots)$ and the interpretations of operators. The following are the truth tables for NFO of $\mathscr{A'}$-\textit{wfs} and FO of $\mathscr{A}$-\textit{wfs}. The final analysis is represented by the darker color column.
\begin{table}[ht!]	\begin{center}
	\begin{tabular}{| c |>{\columncolor[rgb]{0.6,0.6,0.6}} c | c |}	\hline \rowcolor[rgb]{0.6,0.6,0.6}
		$p$&$\fab$&$q$ \\ \hline
		1 & 0 & 1 \\	\hline
		1 & 0 & 0 \\	\hline
		0 & 0 & 1 \\	\hline
		0 & 1 & 0 \\	\hline	
	\end{tabular} 
	\begin{tabular}{| c |>{\columncolor[rgb]{0.6,0.6,0.6}} c | c |}	\hline \rowcolor[rgb]{0.6,0.6,0.6}
		$p$&$\far$&$q$ \\ \hline
		1 & 0 & 1 \\	\hline
		1 & 1 & 0 \\	\hline
		0 & 1 & 1 \\	\hline
		0 & 1 & 0 \\	\hline	
	\end{tabular} 
	\begin{tabular}{| c |>{\columncolor[rgb]{0.6,0.6,0.6}} c | c |}	\hline \rowcolor[rgb]{0.6,0.6,0.6}
		$p$&$\fizq$&$q$ \\ \hline
		1 & 0 & 1 \\	\hline
		1 & 1 & 0 \\	\hline
		0 & 0 & 1 \\	\hline
		0 & 0 & 0 \\	\hline	
	\end{tabular} 
	\begin{tabular}{| c |>{\columncolor[rgb]{0.6,0.6,0.6}} c | c |}	\hline \rowcolor[rgb]{0.6,0.6,0.6}
		$p$&$\oplus$&$q$ \\ \hline
		1 & 0 & 1 \\	\hline
		1 & 1 & 0 \\	\hline
		0 & 1 & 1 \\	\hline
		0 & 0 & 0 \\	\hline	
	\end{tabular} \vs
	\begin{tabular}{| c |>{\columncolor[rgb]{0.6,0.6,0.6}} c | c |}	\hline \rowcolor[rgb]{0.6,0.6,0.6}
		$p$&$\lor$&$q$ \\ \hline
		1 & 1 & 1 \\	\hline
		1 & 1 & 0 \\	\hline
		0 & 1 & 1 \\	\hline
		0 & 0 & 0 \\	\hline	
	\end{tabular} 
	\begin{tabular}{| c |>{\columncolor[rgb]{0.6,0.6,0.6}} c | c |}	\hline \rowcolor[rgb]{0.6,0.6,0.6}
		$p$&$\land$&$q$ \\ \hline
		1 & 1 & 1 \\	\hline
		1 & 0 & 0 \\	\hline
		0 & 0 & 1 \\	\hline
		0 & 0 & 0 \\	\hline	
	\end{tabular} 
	\begin{tabular}{| c |>{\columncolor[rgb]{0.6,0.6,0.6}} c | c |}	\hline \rowcolor[rgb]{0.6,0.6,0.6}
		$p$&$\fd$&$q$ \\ \hline
		1 & 1 & 1 \\	\hline
		1 & 0 & 0 \\	\hline
		0 & 1 & 1 \\	\hline
		0 & 1 & 0 \\	\hline	
	\end{tabular} 
	\begin{tabular}{| c |>{\columncolor[rgb]{0.6,0.6,0.6}} c | c |}	\hline \rowcolor[rgb]{0.6,0.6,0.6}
		$p$&$\fid$&$q$ \\ \hline
		1 & 1 & 1 \\	\hline
		1 & 0 & 0 \\	\hline
		0 & 0 & 1 \\	\hline
		0 & 1 & 0 \\	\hline	
	\end{tabular} \end{center} 
	\end{table} 
\end{defn}

To simplify writing let $\iin$ a primitive symbol that describes ``are \textit{wfs} of''

\begin{defn}[Semantic Parallel] 
$\mathscr{A'} \ \iin \ \Sigma_{NFO}$ is the parallel of $\mathscr{A} \ \iin \ \Sigma_{FO}$ iff
$\mathscr{I} (\mathscr{A'})$ is equal to $\mathscr{I} (\mathscr{A})$ for all values of the atoms in the
final analysis of $\mathscr{A}$ and $\mathscr{A'}$, the parallel is denoted by $\mathscr{A} \parallel \mathscr{A'}$.
\end{defn}

\begin{defn}[Semantic Perpendicularity] 
$\mathscr{A'} \ \iin \ \Sigma_{NFO}$ is the perpendicular of $\mathscr{A} \ \iin \ \Sigma_{FO}$ iff
$\mathscr{I} (\mathscr{A'})$ is equal to $\mathscr{I} (\neg \mathscr{A})$ for all values of the atoms in the
final analysis of $\mathscr{A}$ and $\mathscr{A'}$, the perpendicularity is denoted by $\mathscr{A} \perp \mathscr{A'}$.
\end{defn}

\begin{defn}[Tautology]
$\mathscr{A}$-\textit{wfs} or $\mathscr{A'}$-\textit{wfs} are tautology if the interpretation $\mathscr{I}(\mathscr{A})=1$ or
$\mathscr{I}(\mathscr{A'})=1$ respectively for all values of the final analysis.
\end{defn}

\begin{defn}[Contradiction]
$\mathscr{A}$-\textit{wfs} or $\mathscr{A'}$-\textit{wfs} is a contradiction if the interpretation $\mathscr{I}(\mathscr{A})=0$ or 
$\mathscr{I}(\mathscr{A'})=0$ respectively for all values of the final analysis.
\end{defn}

\begin{prop}Associativity and distributive properties are tautologies  \textsc{\cite{B4}} \textsc{\cite{B7}} in the 0-OL semantic with FO i.e.
\begin{enumerate}[label=$\mathcal{A}_{\arabic*}$, topsep=2pt, partopsep=2pt, itemsep=2pt, parsep=2pt]
\item $\mathscr{I}((p \lor(q \lor r)) \fid ((p \lor q) \lor r))=1$
\item $\mathscr{I}((p \land (q \land r)) \fid ((p \land q) \land r))=1$
\item $\mathscr{I}((p \land (q \lor r)) \fid ((p \land q) \lor(p \land r)))=1$
\item $\mathscr{I}((p \lor(q \land r)) \fid ((p \lor q) \land (p \lor r)))=1$
\end{enumerate}
\end{prop}

\begin{proof} With truth tables can be verified
\begin{enumerate}[label=$\mathcal{A}_{\arabic*}$, topsep=2pt, partopsep=2pt, itemsep=2pt, parsep=2pt]
\item \begin{tabular}{| c |>{\columncolor[rgb]{0.8,0.8,0.8}} c | c |>{\columncolor[rgb]{0.9,0.9,0.9}} c | c |>{\columncolor[rgb]						{0.6,0.6,0.6}} c | c |>{\columncolor[rgb]{0.9,0.9,0.9}} c | c |>{\columncolor[rgb]{0.8,0.8,0.8}} c | c | }
			\hline \rowcolor[rgb]{0.6,0.6,0.6}
			$(p$&$\lor$&$(q$&$\lor$&$r))$&$\fid$&$((p$&$\lor$&$q)$&$\lor$&$r)$ \\ \hline
			1 & 1 & 1 & 1 & 1 & 1 & 1 & 1 & 1 & 1 & 1 \\	\hline
			1 & 1 & 1 & 1 & 0 & 1 & 1 & 1 & 1 & 1 & 0 \\	\hline
			1 & 1 & 0 & 1 & 1 & 1 & 1 & 1 & 0 & 1 & 1 \\	\hline
			1 & 1 & 0 & 0 & 0 & 1 & 1 & 1 & 0 & 1 & 0 \\	\hline
			0 & 1 & 1 & 1 & 1 & 1 & 0 & 1 & 1 & 1 & 1 \\	\hline
			0 & 1 & 1 & 1 & 0 & 1 & 0 & 1 & 1 & 1 & 0 \\	\hline
			0 & 1 & 0 & 1 & 1 & 1 & 0 & 0 & 0 & 1 & 1 \\	\hline
			0 & 0 & 0 & 0 & 0 & 1 & 0 & 0 & 0 & 0 & 0 \\	\hline
		\end{tabular}
\item \begin{tabular}{| c |>{\columncolor[rgb]{0.8,0.8,0.8}} c | c |>{\columncolor[rgb]{0.9,0.9,0.9}} c | c |>{\columncolor[rgb]						{0.6,0.6,0.6}} c | c |>{\columncolor[rgb]{0.9,0.9,0.9}} c | c |>{\columncolor[rgb]{0.8,0.8,0.8}} c | c | }
			\hline \rowcolor[rgb]{0.6,0.6,0.6}
			$(p$&$\land$&$(q$&$\land$&$r))$&$\fid$&$((p$&$\land$&$q)$&$\land$&$r)$ \\ \hline
			1 & 1 & 1 & 1 & 1 & 1 & 1 & 1 & 1 & 1 & 1 \\	\hline
			1 & 0 & 1 & 0 & 0 & 1 & 1 & 1 & 1 & 0 & 0 \\	\hline
			1 & 0 & 0 & 0 & 1 & 1 & 1 & 0 & 0 & 0 & 1 \\	\hline
			1 & 0 & 0 & 0 & 0 & 1 & 1 & 0 & 0 & 0 & 0 \\	\hline
			0 & 0 & 1 & 1 & 1 & 1 & 0 & 0 & 1 & 0 & 1 \\	\hline
			0 & 0 & 1 & 0 & 0 & 1 & 0 & 0 & 1 & 0 & 0 \\	\hline
			0 & 0 & 0 & 0 & 1 & 1 & 0 & 0 & 0 & 0 & 1 \\	\hline
			0 & 0 & 0 & 0 & 0 & 1 & 0 & 0 & 0 & 0 & 0 \\	\hline
		\end{tabular}
\item \begin{tabular}{| c |>{\columncolor[rgb]{0.8,0.8,0.8}} c | c |>{\columncolor[rgb]{0.9,0.9,0.9}} c | c |>{\columncolor[rgb]						{0.6,0.6,0.6}} c | c |>{\columncolor[rgb]{0.9,0.9,0.9}} c | c |>{\columncolor[rgb]{0.8,0.8,0.8}} c | c | c | c | c | }
			\hline \rowcolor[rgb]{0.6,0.6,0.6}
			$(p$&$\lor$&$(q$&$\land$&$r))$&$\fid$&$((p$&$\lor$&$q)$&$\land$&$(p$&$\lor$&$r))$ \\ \hline
			1 & 1 & 1 & 1 & 1 & 1 & 1 & 1 & 1 & 1 & 1 & 1 & 1\\	\hline
			1 & 1 & 1 & 0 & 0 & 1 & 1 & 1 & 1 & 1 & 1 & 1 & 0\\	\hline
			1 & 1 & 0 & 0 & 1 & 1 & 1 & 1 & 0 & 1 & 1 & 1 & 1\\	\hline
			1 & 1 & 0 & 0 & 0 & 1 & 1 & 1 & 0 & 1 & 1 & 1 & 0\\	\hline
			0 & 1 & 1 & 1 & 1 & 1 & 0 & 1 & 1 & 1 & 0 & 1 & 1\\	\hline
			0 & 0 & 1 & 0 & 0 & 1 & 0 & 1 & 1 & 0 & 0 & 0 & 0\\	\hline
			0 & 0 & 0 & 0 & 1 & 1 & 0 & 0 & 0 & 0 & 0 & 1 & 1\\	\hline
			0 & 0 & 0 & 0 & 0 & 1 & 0 & 0 & 0 & 0 & 0 & 0 & 0\\	\hline
		\end{tabular}
\item \begin{tabular}{| c |>{\columncolor[rgb]{0.8,0.8,0.8}} c | c |>{\columncolor[rgb]{0.9,0.9,0.9}} c | c |>{\columncolor[rgb]						{0.6,0.6,0.6}} c | c |>{\columncolor[rgb]{0.9,0.9,0.9}} c | c |>{\columncolor[rgb]{0.8,0.8,0.8}} c | c | c | c | c | }
			\hline \rowcolor[rgb]{0.6,0.6,0.6}
			$(p$&$\land$&$(q$&$\lor$&$r))$&$\fid$&$((p$&$\land$&$q)$&$\lor$&$(p$&$\land$&$r))$ \\ \hline
			1 & 1 & 1 & 1 & 1 & 1 & 1 & 1 & 1 & 1 & 1 & 1 & 1\\	\hline
			1 & 1 & 1 & 1 & 0 & 1 & 1 & 1 & 1 & 1 & 1 & 0 & 0\\	\hline
			1 & 1 & 0 & 1 & 1 & 1 & 1 & 0 & 0 & 1 & 1 & 1 & 1\\	\hline
			1 & 0 & 0 & 0 & 0 & 1 & 1 & 0 & 0 & 0 & 1 & 0 & 0\\	\hline
			0 & 0 & 1 & 1 & 1 & 1 & 0 & 0 & 1 & 0 & 0 & 0 & 1\\	\hline
			0 & 0 & 1 & 1 & 0 & 1 & 0 & 0 & 1 & 0 & 0 & 0 & 0\\	\hline
			0 & 0 & 0 & 1 & 1 & 1 & 0 & 0 & 0 & 0 & 0 & 0 & 1\\	\hline
			0 & 0 & 0 & 0 & 0 & 1 & 0 & 0 & 0 & 0 & 0 & 0 & 0\\	\hline
		\end{tabular}
	\end{enumerate}
\end{proof}

\begin{prop} We can create a $\rho$-assoc and $\rho$-dist properties in the 0-OL semantic with NFO that are contradictions
\begin{enumerate}[label=$\mathcal{A'}_{\arabic*}$, topsep=2pt, partopsep=2pt, itemsep=2pt, parsep=2pt]
\item $\mathscr{I}((p \fab \neg (q \fab r)) \oplus (\neg (p \fab q) \fab r))=0$
\item $\mathscr{I}((p \far \neg (q \far r)) \oplus (\neg (p \far q) \far r))=0$
\item $\mathscr{I}((p \fab \neg (q \far r)) \oplus (\neg (p \fab q) \far \neg (p \fab r)))=0$
\item $\mathscr{I}((p \far \neg (q \fab r)) \oplus (\neg (p \far q) \fab \neg (p \far r)))=0$
\end{enumerate}
\end{prop}

\begin{proof} With truth tables can be verified
\begin{enumerate}[label=$\mathcal{A'}_{\arabic*}$, topsep=2pt, partopsep=2pt, itemsep=2pt, parsep=2pt]
\item \begin{tabular}{| c |>{\columncolor[rgb]{0.8,0.8,0.8}} c | c |>{\columncolor[rgb]{0.9,0.9,0.9}} c | c |>{\columncolor[rgb]						{0.6,0.6,0.6}} c | c |>{\columncolor[rgb]{0.9,0.9,0.9}} c | c |>{\columncolor[rgb]{0.8,0.8,0.8}} c | c | }
			\hline \rowcolor[rgb]{0.6,0.6,0.6}
			$(p$&$\fab$&$\neg(q$&$\fab$&$r))$&$\oplus$&$(\neg(p$&$\fab$&$q)$&$\fab$&$r)$ \\ \hline
			1 & 0 & 1 & 1 & 1 & 0 & 1 & 1 & 1 & 0 & 1 \\	\hline
			1 & 0 & 1 & 1 & 0 & 0 & 1 & 1 & 1 & 0 & 0 \\	\hline
			1 & 0 & 0 & 1 & 1 & 0 & 1 & 1 & 0 & 0 & 1 \\	\hline
			1 & 0 & 0 & 0 & 0 & 0 & 1 & 1 & 0 & 0 & 0 \\	\hline
			0 & 0 & 1 & 1 & 1 & 0 & 0 & 1 & 1 & 0 & 1 \\	\hline
			0 & 0 & 1 & 1 & 0 & 0 & 0 & 1 & 1 & 0 & 0 \\	\hline
			0 & 0 & 0 & 1 & 1 & 0 & 0 & 0 & 0 & 0 & 1 \\	\hline
			0 & 1 & 0 & 0 & 0 & 0 & 0 & 0 & 0 & 1 & 0 \\	\hline
		\end{tabular}
\item \begin{tabular}{| c |>{\columncolor[rgb]{0.8,0.8,0.8}} c | c |>{\columncolor[rgb]{0.9,0.9,0.9}} c | c |>{\columncolor[rgb]						{0.6,0.6,0.6}} c | c |>{\columncolor[rgb]{0.9,0.9,0.9}} c | c |>{\columncolor[rgb]{0.8,0.8,0.8}} c | c | }
			\hline \rowcolor[rgb]{0.6,0.6,0.6}
			$(p$&$\far$&$\neg(q$&$\far$&$r))$&$\oplus$&$(\neg(p$&$\far$&$q)$&$\far$&$r)$ \\ \hline
			1 & 0 & 1 & 1 & 1 & 0 & 1 & 1 & 1 & 0 & 1 \\	\hline
			1 & 1 & 1 & 0 & 0 & 0 & 1 & 1 & 1 & 1 & 0 \\	\hline
			1 & 1 & 0 & 0 & 1 & 0 & 1 & 0 & 0 & 1 & 1 \\	\hline
			1 & 1 & 0 & 0 & 0 & 0 & 1 & 0 & 0 & 1 & 0 \\	\hline
			0 & 1 & 1 & 1 & 1 & 0 & 0 & 0 & 1 & 1 & 1 \\	\hline
			0 & 1 & 1 & 0 & 0 & 0 & 0 & 0 & 1 & 1 & 0 \\	\hline
			0 & 1 & 0 & 0 & 1 & 0 & 0 & 0 & 0 & 1 & 1 \\	\hline
			0 & 1 & 0 & 0 & 0 & 0 & 0 & 0 & 0 & 1 & 0 \\	\hline
		\end{tabular}
\item \begin{tabular}{| c |>{\columncolor[rgb]{0.8,0.8,0.8}} c | c |>{\columncolor[rgb]{0.9,0.9,0.9}} c | c |>{\columncolor[rgb]						{0.6,0.6,0.6}} c | c |>{\columncolor[rgb]{0.9,0.9,0.9}} c | c |>{\columncolor[rgb]{0.8,0.8,0.8}} c | c | c | c | c | }
			\hline \rowcolor[rgb]{0.6,0.6,0.6}
			$(p$&$\fab$&$\neg(q$&$\far$&$r))$&$\oplus$&$(\neg (p$&$\fab$&$q)$&$\far$&$\neg (p$&$\fab$&$r))$ \\ \hline
			1 & 0 & 1 & 1 & 1 & 0 & 1 & 1 & 1 & 0 & 1 & 1 & 1\\	\hline
			1 & 0 & 1 & 0 & 0 & 0 & 1 & 1 & 1 & 0 & 1 & 1 & 0\\	\hline
			1 & 0 & 0 & 0 & 1 & 0 & 1 & 1 & 0 & 0 & 1 & 1 & 1\\	\hline
			1 & 0 & 0 & 0 & 0 & 0 & 1 & 1 & 0 & 0 & 1 & 1 & 0\\	\hline
			0 & 0 & 1 & 1 & 1 & 0 & 0 & 1 & 1 & 0 & 0 & 1 & 1\\	\hline
			0 & 1 & 1 & 0 & 0 & 0 & 0 & 1 & 1 & 1 & 0 & 0 & 0\\	\hline
			0 & 1 & 0 & 0 & 1 & 0 & 0 & 0 & 0 & 1 & 0 & 1 & 1\\	\hline
			0 & 1 & 0 & 0 & 0 & 0 & 0 & 0 & 0 & 1 & 0 & 0 & 0\\	\hline
		\end{tabular}
\item \begin{tabular}{| c |>{\columncolor[rgb]{0.8,0.8,0.8}} c | c |>{\columncolor[rgb]{0.9,0.9,0.9}} c | c |>{\columncolor[rgb]						{0.6,0.6,0.6}} c | c |>{\columncolor[rgb]{0.9,0.9,0.9}} c | c |>{\columncolor[rgb]{0.8,0.8,0.8}} c | c | c | c | c | }
			\hline \rowcolor[rgb]{0.6,0.6,0.6}
			$(p$&$\far$&$\neg(q$&$\fab$&$r))$&$\oplus$&$(\neg (p$&$\far$&$q)$&$\fab$&$\neg (p$&$\far$&$r))$ \\ \hline
			1 & 0 & 1 & 1 & 1 & 0 & 1 & 1 & 1 & 0 & 1 & 1 & 1\\	\hline
			1 & 0 & 1 & 1 & 0 & 0 & 1 & 1 & 1 & 0 & 1 & 0 & 0\\	\hline
			1 & 0 & 0 & 1 & 1 & 0 & 1 & 0 & 0 & 0 & 1 & 1 & 1\\	\hline
			1 & 1 & 0 & 0 & 0 & 0 & 1 & 0 & 0 & 1 & 1 & 0 & 0\\	\hline
			0 & 1 & 1 & 1 & 1 & 0 & 0 & 0 & 1 & 1 & 0 & 0 & 1\\	\hline
			0 & 1 & 1 & 1 & 0 & 0 & 0 & 0 & 1 & 1 & 0 & 0 & 0\\	\hline
			0 & 1 & 0 & 1 & 1 & 0 & 0 & 0 & 0 & 1 & 0 & 0 & 1\\	\hline
			0 & 1 & 0 & 0 & 0 & 0 & 0 & 0 & 0 & 1 & 0 & 0 & 0\\	\hline
		\end{tabular}
\end{enumerate}
\end{proof}
	
\begin{thm}
$\mathscr{A}_1 \perp \mathscr{A'}_1,\mathscr{A}_2 \perp \mathscr{A'}_2,
\mathscr{A}_3 \perp \mathscr{A'}_3,\mathscr{A}_4 \perp \mathscr{A'}_4$
\end{thm}
\begin{proof}
Clearly $\mathscr{A}_1,\mathscr{A}_2,\mathscr{A}_3,\mathscr{A}_4 \ \iin \ \Sigma_{FO}$ and 
$\mathscr{A'}_1,\mathscr{A'}_2,\mathscr{A'}_3,\mathscr{A'}_4 \ \iin \ \Sigma_{NFO}$ also
$\mathscr{I}(\mathscr{A}_1)=\mathscr{I}(\neg \mathscr{A'}_1)=1$ \\
$\mathscr{I}(\mathscr{A}_2)=\mathscr{I}(\neg \mathscr{A'}_2)=1$ \\
$\mathscr{I}(\mathscr{A}_3)=\mathscr{I}(\neg \mathscr{A'}_3)=1$ \\
$\mathscr{I}(\mathscr{A}_4)=\mathscr{I}(\neg \mathscr{A'}_4)=1$ \\
Then by definition of $\perp$\\
$\mathscr{A}_1 \perp \mathscr{A'}_1,\mathscr{A}_2 \perp \mathscr{A'}_2,
\mathscr{A}_3 \perp \mathscr{A'}_3,\mathscr{A}_4 \perp \mathscr{A'}_4$
\end{proof}
\begin{cor}
$\mathscr{A}_1 \parallel \neg \mathscr{A'}_1,\mathscr{A}_2 \parallel \neg \mathscr{A'}_2,
\mathscr{A}_3 \parallel \neg \mathscr{A'}_3,\mathscr{A}_4 \parallel \neg \mathscr{A'}_4$
\end{cor}
\begin{cor}
Let $\mathscr{A} \ \iin \ \Sigma_{FO}$ and $\mathscr{A'} \ \iin \ \Sigma_{NFO}$
\begin{enumerate}[label=$\alph*)$, topsep=2pt, partopsep=2pt, itemsep=2pt, parsep=2pt]
\item $\mathscr{A} \perp \mathscr{A'}$ iff $\mathscr{A} \parallel \neg \mathscr{A'}$ 
\item $\mathscr{A} \parallel \neg \mathscr{A'}$ iff $\neg \mathscr{A} \parallel \mathscr{A'}$ 
\end{enumerate}
\end{cor}

We proceed to create two semantic rules of reemplacement that guarantee a $\rho$-assoc and $\rho$-dist with NFO.

\begin{defn}
Let $\Upsilon$ the property that encrypts $\neg$ monary operator in \textit{wfs} $\mathscr{A'} \ \iin \ \Sigma_{NFO}$ iff
this precedes a parentheses with two $A'$-\textit{wfs} which can be atoms operated by $\fab$ or $\far$ NFOs, but does not change the interpretation.
\begin{itemize}
\item $\Upsilon((p \fab \neg (q \fab r)) \oplus (\neg (p \fab q) \fab r))$ 
		is $(p \fab (q \fab r)) \oplus ((p \fab q) \fab r)$
\item $\Upsilon((p \far \neg (q \far r)) \oplus (\neg (p \far q) \far r))$
		is $(p \far (q \far r)) \oplus ((p \far q) \far r))$
\item $\Upsilon((p \fab \neg (q \far r)) \oplus (\neg (p \fab q) \far \neg (p \fab r))$ is \\
		$((p \fab (q \far r)) \oplus ((p \fab q) \far (p \fab r))$
\item $\Upsilon((p \far \neg (q \fab r)) \oplus (\neg (p \far q) \fab \neg (p \far r))$ is \\
		$((p \far (q \fab r)) \oplus ((p \far q) \fab (p \far r))$
\end{itemize}
\end{defn}

\begin{lem}
If $\mathscr{A} \perp \mathscr{A'}$ we can construct a rule such that changing $\mathscr{I}(\mathscr{A'})$ obtain
$\mathscr{A} \parallel \mathscr{B'}$.
\end{lem}
\begin{proof}
By \textbf{Corolary 2.12} $\mathscr{A} \perp \mathscr{A'}$ can be $\mathscr{A} \parallel \neg \mathscr{A'}$, now
$\neg \mathscr{A'}$ is $\mathscr{B'}$
\end{proof}

\begin{defn}
If $\Psi$ is the property mentioned in \textbf{Lemma 2.14}, this changes the interpretation of the $\mathscr{A'}$ final analysis and converts $\mathscr{A}$ in $\neg \mathscr{A'}$.
As $\rho$-assoc and $\rho$-dist properties with NFO have the final analysis with the $\oplus$ binary operator $\Psi$ only needs to change the interpretation of $\oplus$ operator, call this replacement operator $\updownarrow$.
\begin{itemize}
\item $\Psi(((p \fab (q \fab r)) \oplus ((p \fab q) \fab r))$ is $((p \fab (q \fab r)) \updownarrow ((p \fab q) \fab r))$ 
\item $\Psi(((p \far (q \far r)) \oplus ((p \far q) \far r))$ is $((p \far (q \far r)) \updownarrow ((p \far q) \far r))$
\item $\Psi(((p \fab (q \far r)) \oplus ((p \fab q) \far (p \fab r)))$ is 
$((p \fab (q \far r)) \updownarrow ((p \fab q) \far (p \fab r)))$
\item $\Psi(((p \far (q \fab r)) \oplus ((p \far q) \fab (p \far r)))$ is
$((p \far (q \fab r)) \updownarrow ((p \far q) \fab (p \far r)))$
\end{itemize}
\end{defn}

\begin{thm}
NFO $\updownarrow$ perform the same operations of FO $\fid$
\end{thm}
\begin{proof}
By \textbf{Definition 1.1} $\oplus$ is the negation of $\fid$ and how $\updownarrow$ change the interpretation of $\oplus$ then
$\updownarrow$ perform the same operations of $\fid$
\end{proof}
Now we can change $\updownarrow$ of NFO for $\fid$ of FO.

\begin{cor} The next $\mathscr{A'}$-\textit{wfs} are tautologies
\begin{enumerate}[label=(\alph*), topsep=2pt, partopsep=2pt, itemsep=2pt, parsep=2pt]
\item $\neg (((p \fab \neg (q \fab r)) \oplus (\neg (p \fab q) \fab r)))$
\item $\neg (((p \far \neg (q \far r)) \oplus (\neg (p \far q) \far r)))$
\item $\neg (((p \fab \neg (q \far r)) \oplus (\neg (p \fab q) \far \neg (p \fab r)))$
\item $\neg (((p \far \neg (q \fab r)) \oplus (\neg (p \far q) \fab \neg (p \far r)))$
\end{enumerate}
\end{cor}

In this way we have obtained a $\rho$-assoc and $\rho$-dist of $\mathscr{A}$-\textit{wfs} $\iin \ \Sigma$ with NFO in the semantic of the 0-OL with the application of the rules $\Upsilon$ and $\Psi$.

\begin{note} The \textit{wfs} of \textbf{Corollary 2.17} encrypted by $\Upsilon, \Psi$ are.
\begin{enumerate}[label=(\alph*), topsep=2pt, partopsep=2pt, itemsep=2pt, parsep=2pt]
\item $ (p \fab (q \fab r)) \fid ((p \fab q) \fab r)$
\item $ (p \far (q \far r)) \fid ((p \far q) \far r)$
\item $ (p \fab (q \far r)) \fid ((p \fab q) \far (p \fab r)$
\item $ (p \far (q \fab r)) \fid ((p \far q) \fab (p \far r)$
\end{enumerate}
\end{note}

\begin{note} The rules $\Upsilon$ and $\Psi$ are decryptable.
\end{note}

\section{Sintactical comparison}

In \cite{B7} is defined $\mathscr{L}$-Theory of $\mathscr{HA}$ where $\mathscr{A}$-\textit{f} is a formula, $p, q, \dots$ are symbols

\begin{defn}[$\mathscr{A}$-\textit{f} of $\mathscr{L_{HA}}$]  If $p$-\textit{f}, $q$-\textit{f} of $\mathscr{L_{HA}}$, the following are formulas of $\mathscr{L_{HA}}$
\begin{itemize}
\item $p \fab q$ is $\neg(p \lor q)$ 
\item $p \far q$ is $\neg(p \land q)$ 
\item $p \fizq q$ is $\neg(p \fd q)$ 
\item $p \oplus q$ is $\neg(p \fid q)$ 
\end{itemize}
\end{defn}

\begin{prop} $\mathscr{L_{HA}}$ satisfies
\begin{enumerate}[label=(\alph*), topsep=2pt, partopsep=2pt, itemsep=2pt, parsep=2pt]
\item $\vdash_{\mathscr{L}_{HA}}\neg (((p \fab \neg (q \fab r)) \oplus (\neg (p \fab q) \fab r)))$
\item $\vdash_{\mathscr{L}_{HA}}\neg (((p \far \neg (q \far r)) \oplus (\neg (p \far q) \far r)))$
\item $\vdash_{\mathscr{L}_{HA}}\neg (((p \fab \neg (q \far r)) \oplus (\neg (p \fab q) \far \neg (p \fab r)))$
\item $\vdash_{\mathscr{L}_{HA}}\neg (((p \far \neg (q \fab r)) \oplus (\neg (p \far q) \fab \neg (p \far r)))$
\end{enumerate}
\end{prop}
\begin{proof}
By \textbf{Corollary 2.17} \textit{a,b,c,d} are tautologies, by completeness theorem \cite{B3} \cite{B5} \cite{B7} \textit{a,b,c,d} are theorems of $\mathscr{L_{HA}}$
\end{proof}

\section{Main result}

The next formulas are the result of the paper.

\begin{enumerate}[label=(\alph*), topsep=2pt, partopsep=2pt, itemsep=2pt, parsep=2pt]
\item $\vdash_{\mathscr{L}_{HA}} (p \fab \neg (q \fab r)) \fid (\neg (p \fab q) \fab r))$
\item $\vdash_{\mathscr{L}_{HA}} (p \far \neg (q \far r)) \fid (\neg (p \far q) \far r))$
\item $\vdash_{\mathscr{L}_{HA}} (p \fab \neg (q \far r)) \fid (\neg (p \fab q) \far \neg (p \fab r))$
\item $\vdash_{\mathscr{L}_{HA}} (p \far \neg (q \fab r)) \fid (\neg (p \far q) \fab \neg (p \far r))$
\end{enumerate}

% Bibliografia.
%-----------------------------------------------------------------

\end{document}